\documentclass[12pt]{article}

\usepackage{verbatim}
\usepackage{setspace}
\usepackage{sectsty}
\subsectionfont{\normalsize}
\makeatletter
\renewcommand\section{\@startsection{section}{1}{\z@}%
                                  {-2.0ex \@plus -1ex \@minus -.2ex}%
                                  {2.0ex \@plus.2ex}%
                                  {\normalfont\normalsize\bfseries}}
\makeatother
\usepackage{fullpage}
\usepackage{url}
\usepackage{amsmath}
\usepackage{mathdots}
\renewcommand{\mod}{{\ \rm mod\ }}

\newcommand{\alphaa}[2]{\alpha_{#1#2}}
\newcommand{\alphaaa}[3]{\alpha_{#1#2}^{#3}}
\newcommand{\nn}[2]{n_{#1#2}}

\newcommand{\AAA}[3]{A_{#1#2#3}}
\newcommand{\X}[2]{X_{#1#2}}

\newcommand{\ang}[1]{\langle#1\rangle}

\newcommand{\Z}{{\sf Z}}
\newcommand{\nat}{{\sf N}}
\newcommand{\rat}{{\sf Q}}
\newcommand{\real}{{\sf R}}

\newcommand{\xvec}[1]{\ifcase 3{#1} {\ang {x_1,x_2,x_3} } \else 
\ifcase 4{#1} {\ang{x_1,x_2,x_3,x_4}} \else {\ang {x_1,\ldots,x_{#1}}}\fi\fi}
\newcommand{\yvec}[1]{\ifcase 3{#1} {\ang {y_1,y_2,y_3} } \else 
\ifcase 4{#1} {\ang{y_1,y_2,y_3,y_4}} \else {\ang {y_1,\ldots,y_{#1}}}\fi\fi}
\newcommand{\zvec}[1]{\ifcase 3{#1} {\ang {z_1,z_2,z_3} } \else 
\ifcase 4{#1} {\ang{z_1,z_2,z_3,z_4}} \else {\ang {z_1,\ldots,z_{#1}}}\fi\fi}
\newcommand{\vecc}[2]{\ifcase 3{#2} {\ang { {#1}_1,{#1}_2,{#1}_3 } } \else
\ifcase 4{#1} {\ang { {#1}_1,{#1}_2,{#1}_3,{#1}_{4} } }
\else {\ang { {#1}_1,\ldots,{#1}_{#2}}}\fi\fi}
\newcommand{\veccd}[3]{\ifcase 3{#2} {\ang { {#1}_{{#3}1},{#1}_{{#3}2},{#1}_{{#3}3} } } \else
\ifcase 4{#1} {\ang { {#1}_{{#3}1},{#1}_{{#3}2},{#1}_{#3}3},{#1}_{{#3}4} }
\else {\ang { {#1}_{{#3}1},\ldots,{#1}_{{#3}{#2}}}}\fi\fi}
\newcommand{\veccz}[2]{\ifcase 3{#2} {\ang { {#1}_0,{#1}_2,{#1}_3 } } \else
\ifcase 4{#1} {\ang { {#1}_0,{#1}_2,{#1}_3,{#1}_{4} } }
\else {\ang { {#1}_0,\ldots,{#1}_{#2}}}\fi\fi}
\newcommand{\xve}[1]{\ifcase 3{#1} {x_1,x_2,x_3} \else 
\ifcase 4{#1} {x_1,x_2,x_3,x_4} \else {x_1,\ldots,x_{#1}}\fi\fi}
\newcommand{\yve}[1]{\ifcase 3{#1} {y_1,y_2,y_3} \else 
\ifcase 4{#1} {y_1,y_2,y_3,y_4} \else {y_1,\ldots,y_{#1}}\fi\fi}
\newcommand{\zve}[1]{\ifcase 3{#1} {z_1,z_2,z_3} \else 
\ifcase 4{#1} {z_1,z_2,z_3,z_4} \else {z_1,\ldots,z_{#1}}\fi\fi}
\newcommand{\ve}[2]{\ifcase 3#2 {{#1}_1,{#1}_2,{#1}_3} \else
\ifcase 4#2 {{#1}_1,{#1}_2,{#1}_3,{#1}_{4}}
\else {{#1}_1,\ldots,{#1}_{#2}}\fi\fi}
\newcommand{\ved}[3]{\ifcase 3#2 {{#1}_{{#3}1},{#1}_{{#3}2},{#1}_{{#3}3}} \else
\ifcase 4#2 {{#1}_{{#3}1},{#1}_{{#3}2},{#1}_{{#3}3},{#1}_{{#3}4}}
\else {{#1}_{{#3}1},\ldots,{#1}_{{#3}{#2}}}\fi\fi}
\newcommand{\fuve}[3]{
\ifcase 3#2
{{#3}({#1}_1),{#3}({#1}_2,{#3}({#1}_3)} \else
\ifcase 4#2
{{#3}({#1}_1),{#3}({#1}_2),{#3}({#1}_3),{#3}({#1}_4)}
\else
{{#3}({#1}_1),\ldots,{#3}({#1}_{#2})}\fi\fi}
\newcommand{\setmathchar}[1]{\ifmmode#1\else$#1$\fi}
\newcommand{\vlist}[2]{%
	\setmathchar{%
		\compound#2\one{#2}\two
		\ifcompound
			({#1}_1,\ldots,{#1}_{#2})
		\else
			\ifcat N#2
				({#1}_1,\ldots,{#1}_{#2})
			\else
				\ifcase#2
					({#1}_0)\or
					({#1}_1)\or
					({#1}_1,{#1}_2)\or 
					({#1}_1,{#1}_2,{#1}_3)\or
					({#1}_1,{#1}_2,{#1}_3,{#1}_4)\else 
					({#1}_1,\ldots,{#1}_{#2})
				\fi
			\fi
		\fi}}

\newif\ifcompound
\def\compound#1\one#2\two{%
	\def\one{#1}
	\def\two{#2}
	\if\one\two
		\compoundfalse
	\else
		\compoundtrue
	\fi}

\newcommand{\xwe}[1]{\ifcase 3{#1} {x_1\wedge x_2\wedge x_3} \else 
\ifcase 4{#1} {x_1\wedge x_2\wedge x_3\wedge x_4} \else {x_1\wedge \cdots \wedge
x_{#1}}\fi\fi}
\newcommand{\we}[2]{\ifcase 3#2 {\ang { {#1}_1\wedge {#1}_2\wedge {#1}_3 } } \else
\ifcase 4{#1} {\ang { {#1}_1\wedge {#1}_2\wedge {#1}_3\wedge {#1}_{4} } }
\else {\ang { {#1}_1\wedge \cdots\wedge {#1}_{#2}}}\fi\fi}

\newcommand{\st}{\mathrel{:}}

\newcommand{\s}[1]{\s_{#1}}

\newcommand{\monus}{\;\raise.5ex\hbox{{${\buildrel
    \ldotp\over{\hbox to 6pt{\hrulefill}}}$}}\;}

\newcommand{\infinity}{\infty}

\newcounter{savenumi}

\newtheorem{theoremfoo}{Theorem}[section] 
\newenvironment{theorem}{\pagebreak[1]\begin{theoremfoo}}{\end{theoremfoo}}

\newtheorem{lemmafoo}[theoremfoo]{Lemma}
\newenvironment{lemma}{\pagebreak[1]\begin{lemmafoo}}{\end{lemmafoo}}
\newtheorem{conjecturefoo}[theoremfoo]{Conjecture}

\newtheorem{conventionfoo}[theoremfoo]{Convention}

\newtheorem{porismfoo}[theoremfoo]{Porism}

\newtheorem{gamefoo}[theoremfoo]{Game}

\newtheorem{corollaryfoo}[theoremfoo]{Corollary}
\newenvironment{corollary}{\pagebreak[1]\begin{corollaryfoo}}{\end{corollaryfoo}}

\newtheorem{openfoo}[theoremfoo]{Open Problem}

\newtheorem{exercisefoo}{Exercise}

\newcommand{\fig}[1] 
{
 \begin{figure}
 \begin{center}
 \input{#1}
 \end{center}
 \end{figure}
}

\newtheorem{potanafoo}[theoremfoo]{Potential Analogue}

\newtheorem{notefoo}[theoremfoo]{Note}
\newenvironment{note}{\pagebreak[1]\begin{notefoo}\rm}{\end{notefoo}}

\newtheorem{notabenefoo}[theoremfoo]{Nota Bene}

\newtheorem{nttn}[theoremfoo]{Notation}
\newenvironment{notation}{\pagebreak[1]\begin{nttn}\rm}{\end{nttn}}

\newtheorem{empttn}[theoremfoo]{Empirical Note}

\newtheorem{examfoo}[theoremfoo]{Example}

\newtheorem{dfntn}[theoremfoo]{Definition}

\newtheorem{propositionfoo}[theoremfoo]{Proposition}

\newenvironment{proof}
    {\pagebreak[1]{\narrower\noindent {\bf Proof:\quad\nopagebreak}}}{\QED}

\newcommand{\yyskip}{\penalty-50\vskip 5pt plus 3pt minus 2pt}
\newcommand{\blackslug}{\hbox{\hskip 1pt
        \vrule width 4pt height 8pt depth 1.5pt\hskip 1pt}}
\newcommand{\QED}{{\penalty10000\parindent 0pt\penalty10000
        \hskip 8 pt\nolinebreak\blackslug\hfill\lower 8.5pt\null}
        \par\yyskip\pagebreak[1]}

\newcommand{\BBB}{{\penalty10000\parindent 0pt\penalty10000
        \hskip 8 pt\nolinebreak\hbox{\ }\hfill\lower 8.5pt\null}
        \par\yyskip\pagebreak[1]}

\newtheorem{factfoo}[theoremfoo]{Fact}

\newenvironment{block}{\begin{list}{\hbox{}}{\leftmargin 1em
    \itemindent -1em \topsep 0pt \itemsep 0pt \partopsep 0pt}}{\end{list}}

\begin{document}

\newcommand{\KN}{K_{\nat}}
\newcommand{\NRE}{\hbox{NUM-RED-EDGES\ }}
\newcommand{\NBE}{\hbox{NUM-BLUE-EDGES\ }}
\newcommand{\RED}{\hbox{RED\ }}
\newcommand{\BLUE}{\hbox{BLUE\ }}
\newcommand{\REDns}{\hbox{RED}}
\newcommand{\BLUEns}{\hbox{BLUE}}

\title{Three Results on Making Change (An Exposition)}

\author{
{William Gasarch}
\thanks{
University of Maryland,
Dept. of Computer Science,
	College Park, MD\ \ 20742,
\texttt{gasarch@cs.umd.edu}
}
\and
{Naveen Raman}
\thanks{
Richard Montgomery High School,
	Rockville, MD\ \ 20850
\texttt{dsfan414@gmail.com}
}
}

\date{}

\maketitle

\begin{abstract}
Let $a_1,\ldots,a_L$ be relatively prime.
We think of them as coin denominations.
Let $M=LCM(a_1,\ldots,a_L)$ and 
let $CH(n)$ be the number of ways to make change of $n$ cents.
We show there is an
{\it exact} piece wise formula for 
$CH(n)$.
The pieces are polynomials that 
depend on $n\mod M$.
We show that many of the pieces agree
on all but the constant term.
These results are not new; however, our treatment is self-contained, unified, and elementary.
\end{abstract}
 
\section{Introduction}

Throughout this paper we let:
\begin{enumerate}
\item
$a_1,a_2,\ldots,a_L$ be coin denominations.
Assume you have an unlimited number of each coin.
They need not be distinct. Think of having red nickels and blue nickels.
\item
$M=LCM(a_1,\ldots,a_L)$.
\item
$M'= LCM(GCD(a_1,a_2),GCD(a_1,a_3),\ldots,GCD(a_{L-1},a_L))$.
\end{enumerate}

\begin{notation}
If $a_1,\ldots,a_L$ are given then
$CH(n)$ is the number of ways to make change of $n$ cents.
Sylvester called $CH(n)$ {\it the denumerant}.
\end{notation}

Determining $CH(n)$ is known as {\it the problem of finding the coefficients of 
the Sylvester denumerant}.
It is related to the well known Frobenius problem: {\it What is the largest $n$ such that $CH(n)=0$?}
Modern papers on this topic tend to use advanced mathematics.
We list some of the papers~\cite{agnden,alonden,balden,beckfrob,bellden,komfrob,losden,serden}
and some of the books~\cite{alffrob,beckrobins,comtet,riordan} where the problem is discussed.

We obtain an {\it exact} piece wise formula for $CH(n)$ and then refine it.
Our results are not new; however, 
our treatment is self-contained, unified, and elementary treatment.
We include the polynomials for several
coin sets in the Appendix and make
some observations and conjectures.

Our results begin with the following premise:
$\{a_1,\ldots,a_L\}$ is a set of coin denominations that are relatively prime 
with $M$, $M'$ as above. Note that if the coin set is $\{1,5,10,25\}$ then
$M=50$ and $M'=5$. This is typical in that $M'$ is usually much less than $M$.

Our first result is that 
there exist $h_0,h_1,\ldots,h_{M-1}\in \rat[x]$ of degree $L-1$ such that 
$$CH(n) = h_{n \mod M}(n).$$
Bell~\cite{bellden}  attributes this result to 
Sylvester and Cayley 
and refers the reader to Dickson~\cite{dickson} (vol 2) for the history of
denumerants up to 1919.
Bell~\cite{bellden} gave a proof that is simpler than the proof of
Sylvester and Cayley. Our proof is similar to Bell's.

Our second results shows that if you ignore the constant term then many of the
polynomials are identical.
Keep in mind that $M'$ is usually much less than $M$.
We show that there there exist $h_0',\ldots,h_{M'-1}'\in \rat[x]$ of degree $L-1$ and
rationals $b_0,\ldots,b_{M-1}$ such that 
$$CH(n) = h_{n\mod {M'}}'(n) + b_{n\mod  M}.$$
This can be derived from Theorem 1.7 (page 15) of the book by
Beck and Robins~\cite{beckrobins} and probably from other formulas for
$CH(n)$ as well. Our proof is simpler than theirs and may be new.

Our third result is that
$$CH(n)=\frac{n^{L-1}}{(L-1)!a_1a_2\cdots a_L} + O(n^{L-2}).$$
This result is attributed to Schur by 
Riordan~\cite{riordan}, Wilf~\cite{wilfgen}, 
and all of the papers and books cited above that mention it.
Our proof is similar to the one in Wilf's book on
generating functions~\cite{wilfgen}. After we prove this we will give a geometric
interpretation.

We then obtain, as a corollary, three theorems that are similar to those stated
above; however, they apply to {\it any} coin set $\{a_1,\ldots,a_L\}$.

\section{Needed Lemmas}\label{se:lemmas}

We obtain the Taylor expansion for $\frac{1}{(1-x)^L}$ via combinatorics,
not calculus.

\begin{lemma}\label{le:taylor}
For all $L$, $\frac{1}{(1-x)^L} = \sum_{n=0}^\infinity \binom{L-1+n}{L-1}x^n$.
\end{lemma}

\begin{proof}
We rewrite this as
$$(1+x+x^2+\cdots )^L = \sum_{n=0}^\infinity \binom{L-1+n}{L-1}x^n.$$

Let $S(L,n)$ be the 
number of solutions of $x_1 + \cdots + x_L=n$ where $x_i\ge 0$.
Clearly the coefficient of $x^n$ of the LHS is $S(L,n)$.
By viewing $S(L,n)$  as the number of ways of permuting $n$ dots and $L-1$ bars
we see that $S(L,n)=\binom{L-1+n}{L-1}.$
Hence the LHS and the RHS are the same.
\end{proof}

We leave the following lemma to the reader.

\begin{lemma}\label{le:prim}
If $\zeta^a=1$ then there exists $d$ such that
$\zeta$ is a primitive $d$th root of unity and $d$ divides $a$.
\end{lemma}

\begin{lemma}\label{le:roots}
Let $a_1,\cdots,a_L$ be relatively prime.
Let $g(x) = (x^{a_1}-1)\cdots(x^{a_L}-1)$. 
When $g(x)$ is factored completely into
linear terms the factor $(x-1)$ occurs $L$ times and all of the other linear
factors occur $\le L-1$ times.
\end{lemma}

\begin{proof}
Let $\zeta$ be a root of $g(x)$. 
We are concerned with the multiplicity of $\zeta$.
By Lemma~\ref{le:prim} $\zeta$ is a primitive $d$th root of unity
where $d$ divides some $a_i$. We denote this $d$ by $d_\zeta$.
The multiplicity of $\zeta$ is $|\{ 1\le j\le L \st d_\zeta | a_j \}|$.
Since the $a_i$'s are relatively prime the only $\zeta$ with
$|\{ 1\le j\le L \st d_\zeta | a_j \}|=L$ is $\zeta=1$.
\end{proof}

\begin{lemma}\label{le:gcd}
Let $a_1,a_2$ be integers and $\zeta$ be a complex number.
If $\zeta^{a_1}=1$ and $\zeta^{a_2}=1$ 
then $\zeta^{GCD(a_1,a_2)}=1$.
\end{lemma}

\begin{proof}
By Lemma~\ref{le:prim} $\zeta$ is a primitive $d$th root of unity where
$d$ divides $a_1$ and $a_2$. Clearly $d$ divide $GCD(a_1,a_2)$.
Hence $\zeta^{GCD(a_1,a_2)}=1$.
\end{proof}

\begin{lemma}\label{le:interpolate}
Let $f$ be a polynomial of degree $L-1$.
If there are $L$ rationals $r$ such that
$f(r)$ is rational then all of the coefficients of $f$ are rational.
\end{lemma}

\begin{proof}
Assume $r_1,\ldots,r_{L}$ are rational and
$f(r_1),\ldots,f(r_{L})$ are rational.

Let 
$h_j(x) = \prod_{i=1,i\ne j}^{L} \frac{x-r_i}{r_j-r_i}$.
Note that
(1) for all $x\in \{r_1,\ldots,r_{L}\}- \{r_j\}$, $h_j(x)=0$,
(2) $h(r_j)=1$, and (3) $h_j$ is a polynomial over the rationals of degree $L-1$.

Let $F(x)=\sum_{j=1}^{L} f(r_j) h_j(x).$
Clearly, for all $1\le i\le L$, $F(r_j)=f(r_j)$.
Hence $F$ and $f$ are polynomials of degree $L-1$ that agree
on $L$ points, so $f=F$. Since $F$ has rational coefficients,
$f$ has rational coefficients.
\end{proof}

\begin{note}
The above proof is based on a well-known technique, called Lagrange interpolation, to
find a polynomial that goes through a given set of points.
\end{note}

\section{Main Theorem}\label{se:main}

\begin{theorem}\label{th:main}
Let $a_1,\ldots,a_L\in\nat$ be relatively prime.
Let $M=LCM(a_1,\ldots,a_L)$ and 
$M'= LCM(GCD(a_1,a_2),GCD(a_1,a_3),\ldots,GCD(a_{L-1},a_L))$.
\begin{enumerate}
\item
There exists $h_0,h_1,\ldots,h_{M-1}\in \rat[x]$ of degree $L-1$ such that $CH(n) = h_{n \mod M}(n)$.
\item
There exists $h_0',\ldots,h_{M'-1}'\in \rat[x]$ of degree $L-1$, and rationals $b_0,\ldots,b_{M-1}$
such that 
$CH(n) = h_{n\mod {M'}}'(n) + b_{n\mod  M}.$
\item
$$CH(n)=\frac{n^{L-1}}{(L-1)!a_1a_2\cdots a_L} + O(n^{L-2}).$$
\end{enumerate}
\end{theorem}

\begin{proof}

The value of $CH(n)$ is the coefficient of $x^n$ in 
\[
\begin{array}{rl}
f(x)& =(1+x^{a_1} + x^{2a_1} + \cdots )(1 + x^{a_2} + x^{2a_2} + \cdots)\cdots(1+x^{a_L} + x^{2a_L} + \cdots)\cr
    & = \frac{1}{ (1-x^{a_1}) (1-x^{a_2}) \cdots (1-x^{a_L}) }.\cr
\end{array}
\]

Assume $a_1\le \cdots \le a_L$ and $i_o$ is such that $a_{i_o}\ge 2$. (If no such $i_o$ exists then $(\forall n)[CH(n)=1]$ and
our theorem is trivially true.)
For all $i_o\le i\le L$, $1\le j\le a_i-1$,
let $\alphaa i j$
be the $j$th $a_i$th 
root of unity (we think of 1 as being
the 0th root of unity).
Let $\nn i j$
be the number of times the factor $(1-\alphaa i j x)$ appears in
$(1-x^{a_1}) (1-x^{a_2}) \cdots (1-x^{a_L})$. 
Since $a_{i_o}\ge 2$ none of the $\alphaa i j$ are 1.
This will be important in the proof of part 3.

We rewrite $f(x)$ using partial fractions and Lemma~\ref{le:taylor} to obtain
$$f(x) = \frac{1}{(1-x)^L \prod_{i=i_o}^L \prod_{j=1}^{a_i-1} (1-\alphaa i j x )^{\nn i j}}
=
\sum_{i=i_o}^L \frac{A_i}{(1-x)^i}
+
\sum_{i=i_o}^L \sum_{j=1}^{a_i-1} \sum_{k=1}^{\nn i j} \frac{\AAA  i j k }{(1-\alphaa i j x )^k}
$$

$$
=
\sum_{i=i_o}^L
\sum_{n=0}^\infinity  A_i\binom{n+i-1}{i-1} x^n
+
\sum_{i=i_o}^L 
\sum_{j=1}^{a_i-1} 
\sum_{k=1}^{\nn i j}
\sum_{n=0}^\infinity \AAA  i j k \binom{n+k-1}{k-1} \alphaaa i j n x^n
$$

$$
=
\sum_{n=0}^\infinity  
\biggl ( 
\sum_{i=i_o}^L
A_i\binom{n+i-1}{i-1}
+
\sum_{i=i_o}^L 
\sum_{j=1}^{a_i-1} 
\sum_{k=1}^{\nn i j}
\AAA  i j k \binom{n+k-1}{k-1} \alphaaa i j n
\biggr ) 
 x^n.
$$

Hence

$$
CH(n)=
\sum_{i=i_o}^L
A_i\binom{n+i-1}{i-1}
+
\sum_{i=i_o}^L 
\sum_{j=1}^{a_i-1} 
\sum_{k=1}^{\nn i j}
\AAA  i j k \binom{n+k-1}{k-1} \alphaaa i j n.
$$

By Lemma~\ref{le:roots} $\nn i j \le L-1$. 
Hence we can write $CH(n)$ as 
$\sum_{e=0}^{L-1} COE(n,e)n^e$ 
where the $COE(n,e)$ are 
functions of the $\alphaaa i j n$.

\bigskip

\noindent
1) Since $\alphaa i j$ is an $a_i$th root of unity, 
$\alphaaa i j n = \alphaaa i j {n \mod M}$.
Hence, for all $0\le e\le L-1$,  $COE(n,e)=COE(n\mod M,e)$.
Therefore the coefficients only depend on $n \mod M$.
For $0\le r\le M-1$ let

$$
h_r(n)=
\sum_{i=i_o}^L
A_i\binom{n+i-1}{i-1}
+
\sum_{i=i_o}^L 
\sum_{j=1}^{a_i-1} 
\sum_{k=1}^{\nn i j}
\AAA  i j k \binom{n+k-1}{k-1} \alphaaa i j r
=
\sum_{e=0}^{L-1} COE(r,e)n^e.
$$

Clearly $h_r$ is a polynomial in $n$ of degree $L-1$ and
$CH(n) = h_{n\mod M}(n)$. Since there is an infinite number of $n\in\nat$ 
(namely all $n\equiv r \pmod M$) such that $h_r(n)\in\nat$,
by Lemma~\ref{le:interpolate} the coefficients
of $h_r$ are rational numbers. Hence $h_r(x)\in\rat[x]$.

\bigskip

\noindent
2) For  $0\le r\le M-1$ let

$$
h_r'(n)=\sum_{i=i_o}^L A_i\binom{n+i-1}{i-1}+\sum_{i=i_o}^L\sum_{j=1}^{a_i-1}\sum_{k=2}^{\nn i j}\AAA i j k \binom{n+k-1}{k-1} \alphaaa i j r.
$$

Note that $h_r(n)$ and $h_r'(n)$ only differ with regard to whether $k$ starts at 1 or 2. 
For $0\le r\le M-1$ let

$$b_r=h_r(n)-h_r'(n) = \sum_{i=i_o}^L\sum_{j=1}^L \AAA i j 1 \alphaaa i j r.$$

Clearly the $b_r$'s are constants (we later show they are rational) and 

$$CH(n) = h_{n \mod M}(n)= h_{n \mod M'}'(n) + b_{n \mod M}.$$

For $e\ge 1$, the coefficient of $n^e$ in both $h_r(n)$ and $h_r'(n)$ are the same.
We need to show that, for $e\ge 1$, $COE(n,e)=COE(n \mod {M'} ,e)$.
Let $e\ge 1$. 
Let $\X k e$
be such that $\binom{n+k-1}{k-1} = \sum_{e=0}^{k-1} \X k e n^e$.
Then

$$
COE(n,e)=
\sum_{i=i_o}^L
A_i \X i e
+
\sum_{i=i_o}^L 
\sum_{j=1}^{a_i-1} 
\sum_{k=2}^{\nn i j}
\AAA  i j k \X k e \alphaaa i j n
$$

Fix $i,j$. If $\nn i j \le 1$ the there is no $k$ with $2\le k\le \nn i j$; therefore we assume 
$\nn i j \ge 2$. So the term $(1-\alphaa i j x)$ 
appears at least twice when factoring 
$(1-x^{a_1}) \cdots (1-x^{a_L})$.
Therefore there exists $i'\ne i$ such that $\alphaa i j$
is an $a_{i'}$th root of unity. Since $\alphaa i j$ is also an $a_i$th root 
of unity,
by Lemma~\ref{le:gcd}, $\alphaa i j$ is a $d$th root of unity
where $d=GCD(a_i,a_{i'})$. 
Since $d$ divides $M'$, $\alphaaa i j {M'} = 1$,
hence $\alphaaa i j n = \alphaaa i j {n \mod {M'}}$.
Therefore
$$
COE(n,e)=
\sum_{i=i_o}^L
A_i \X i e
+
\sum_{i=i_o}^L 
\sum_{j=1}^{a_i-1} 
\sum_{k=2}^{\nn i j}
\AAA  i j k \X k e \alphaaa i j {n \mod {M'}}
$$
which clearly only depends on $n \mod {M'}$.

Fix $0\le r \le M'-1$ and $0\le s\le M-1$ such that there is an 
infinite number of $n\in\nat$ with 
$n\equiv r \pmod {M'}$ and $n\equiv s \pmod M$. 
Hence, for an infinite number of $n\in \nat$,
$h_r'(n)+b_s=CH(n)\in\nat$. By Lemma~\ref{le:interpolate} $h_r'(x)\in \rat[x]$ and the $b_s$'s are rationals.

\bigskip

\noindent
3) 
$$CH(n) = 
\sum_{i=i_o}^L
A_i\binom{n+i-1}{i-1}
+
\sum_{i=i_o}^L 
\sum_{j=1}^{a_i-1} 
\sum_{k=1}^{\nn i j}
\AAA  i j k \binom{n+k-1}{k-1} \alphaaa i j n.
=
A_L\binom{n+L-1}{L-1}.$$

We find $A_L$.

$$
\frac{1}{ (1-x^{a_1}) (1-x^{a_2}) \cdots (1-x^{a_L}) } =
\sum_{i=i_o}^L \frac{A_i}{(1-x)^i}
+
\sum_{i=i_o}^L \sum_{j=1}^{a_i-1} \sum_{k=1}^{\nn i j} \frac{\AAA  i j k }{(1-\alphaa i j x )^k}.
$$

Multiply both sides by $(1-x)^L$ to get

$$
\frac{(1-x)^L}{ (1-x^{a_1}) (1-x^{a_2}) \cdots (1-x^{a_L}) }=
A_L+\sum_{i=i_o}^{L-1} A_i (1-x)^{L-i} + 
\sum_{i=i_o}^L \sum_{j=1}^{a_i-1} \sum_{k=1}^{\nn i j}
 \frac{\AAA  i j k (1-x^L)}{(1-\alphaa i j x )^k}.
$$

The left hand side can be rewritten as

$$
\frac{1}{
(1+x+x^2+\cdots+x^{a_1-1})
(1+x+x^2+\cdots+x^{a_2-1})
\cdots
(1+x+x^2+\cdots+x^{a_L-1})
}
.
$$

As $x$ approaches 1 (from the left), the LHS approaches
$\frac{1}{a_1a_2\cdots a_L}$.
Since for all $i,j$,  $\alphaa i j \ne 1$, as $x$ approaches 1,
the RHS approaches $A_L$.
Hence $A_L=\frac{1}{a_1a_2\cdots a_L}$ and 
$COE(n,L-1)= \frac{1}{(L-1)!a_1a_2\cdots a_L}$.

%
%
%
\end{proof}

An equivalent definition of $CH(n)$ is the
number of integer points in the  set

$$
P_n=\{(x_1,\ldots,x_L)\st \hbox{ all $x_i\ge 0$  and } \sum_{i=1}^L a_ix_i = n \}.
$$

The quantity $\frac{1}{(L-1)!a_1a_2\cdots a_L}$ is the volume of $P_1$. 
Hence Theorem~\ref{th:main}.3 says that the number of integer points in $P_n$
is approximately $VOL(P_1)n^{L-1}$.
Counting the number of integer points in a convex polytope, including
the application to coin problems, is 
studied by Beck and Robins~\cite{beckrobins}.

The following is an easy corollary of Theorem~\ref{th:main}.

\begin{corollary}
Let $a_1,\ldots,a_L$ have greatest common divisor $d$.
Let $M=LCM(a_1/d,\ldots,a_L/d)$ and
$M'= LCM(GCD(a_1/d,a_2/d),GCD(a_1/d,a_3/d),\ldots,GCD(a_{L-1}/d,a_L/d))$.
\begin{enumerate}
\item
If $n\not\equiv 0 \pmod d$ then $CH(n)=0$.
\item
There exists $h_0,h_1,\ldots,h_{M-1}\in \rat[x]$ of degree $L-1$ such that 
if $n\equiv 0 \pmod d$ then 
$CH(n) = h_{n \mod M}(n)$.
\item
There exists $h_0',\ldots,h_{M'-1}'\in \rat[x]$ of degree $L-1$, and rationals  
$b_0,\ldots,b_{M-1}\in \rat$,
such that 
if $n\equiv 0 \pmod d$ then 
$CH(n) = h_{n\mod {M'}}'(n) + b_{n\mod  M}.$
\item
If $CH(n)$ is restricted to $n\equiv 0 \pmod d$ then 
$$CH(n)=\frac{n^{L-1}d^L}{(L-1)!a_1a_2\cdots a_L} + O(n^{L-2}).$$
\end{enumerate}
\end{corollary}

\section{Examples and Conjectures}\label{se:ex}

In the Appendices we present, for a variety of coin sets, $M$, $M'$,
$h_{0\le r\le M}$, $h_{0\le r\le M'}'$, and upper/lower bounds on the $b_i$'s.
When calculating $M'$ 
we omit the pairs of the form $GCD(1,a_j)$ since $GCD(1,a_i)=1$.
For $h_r'$ we take the version with 0 constant term.
We obtained the polynomials via Lagrange interpolation. 
In this section we describe the results and what they might mean.

Let the coin set be $\{1,5,10,25\}$, so that $M=50$ and $M'=5$.
In Appendix~\ref{ap:h151025} we have the polynomials $h_{0\le r\le 49}$.
Note that
(1) if $r_1\equiv r_2 \pmod 5$ then $h_{r_1}$ and $h_{r_2}$ agree on all the coefficients except the constant term,
and  (2) all of the leading coefficients are the same. This is predicted by Theorem~\ref{th:main}.
Also note that 
(1) all of the coefficients are positive,
(2) for all coefficients $c$,
$2(L-1)a_1\cdots a_Lc\in\nat$, and 
(3) the $b_i$'s are small.
Do (1), (2), (3)  hold for all coin sets?

\subsection{Are the Coefficients Always Positive?}

We refer to the statement 

{\it for all coin sets all of the coefficients of the $h$-polynomials associated to them are positive} 

\noindent
as (1).

Clearly (1) does not always hold: if a coin set has $a_1\ne 1$ then $CH(1)=0$ so
some coefficient of $h_1$ has to be negative. In 
Appendices~\ref{ap:h234} 
and \ref{ap:h356} 
we present the polynomials
for the coin sets $\{2,3,4\}$ and $\{3,5,6\}$. For $\{2,3,4\}$ three of the polynomials
have a negative constant term. For $\{3,5,6\}$ eleven of the polynomials have a negative
constant term. All of the non-constant terms have positive coefficients.

Does (1) hold if $a_1=1$?
Alas no.
Of the 138 polynomials for the coin set $\{1,4,6,11\}$,
three of them have a negative constant term. We present these
three polynomials in Appendix~\ref{ap:h14611}.
For all of the polynomials,
all of the non-constant terms 
have positive coefficients.

Does (1) hold if we only look at the non-constant terms?
If we allow a coin denomination to 
appear twice then no.
In Appendix~\ref{ap:h1191920} we present the
polynomials for the coin set $\{1,19,19,20\}$ that have negative linear term.
Of the 380 total polynomials there are 60 (or 3/19) that have a negative linear term.
We also have the following empirical results, which we do not give the polynomials for:
1/7 of the polynomials for $(1,21,21,22)$ have a negative linear term.

Based on our empirical evidence and talking to Matthias Beck and Michelle Vergne (experts in the field)
we have the following conjectures.
\begin{enumerate}
\item
If $a_1,a_2,\ldots,a_L$ are relatively prime then all of the associated polynomials 
have positive coefficients except possibly the constant term.
(It might be easier to prove the $a_1=1$ case.)
\item
If $a_1,a_2,\ldots,a_L$ are relatively prime and $a_1=1$
then all of the associated polynomials
have positive coefficients.
\item
(Michelle Vergne emailed us this conjecture)
For $x$ large and $x<y$ some of the associated polynomials to $(1,x,x,y)$ will
have a negative linear term.
\end{enumerate}

\subsection{Is $2(L-1)a_1\cdots a_Lc$ Always an Integer?}

We refer to the statement 

{\it for the coin set $\{a_1,\ldots,a_L\}$, for all coefficients $c$ of the $h_r$'s, 
$2(L-1)a_1\cdots a_Lc\in\Z$}

\noindent
as (2).

Statement (2) holds for all of the coin sets we have looked at.
There is a known theorem which may be relevant here. We describe it.

A {\it convex rational polytope} is an intersection of
halfspaces such that all of the corner points have rational 
coordinates. 
Recall that $CH(n)$ is the number of integer points in the 
convex rational polytope

$$
P_n=\{(x_1,\ldots,x_L)\st \hbox{ all $x_i\ge 0$  and } \sum_{i=1}^L a_ix_i = n \}.
$$

In Beck and Robins~\cite{beckrobins}
Theorem 3.20 (page 80) states (roughly) that the number of integer points in
a parameterized convex rational polytope is
a piecewise polynomial.
Their Exercise 3.33 (Page 87) states that
for $L$-dimensional rational polytopes in $\real^L$,
for all coefficients  $c$ of those polynomials, $L!c\in\Z$.
Our $P_n$ is not $L$-dimensional and hence their Exercise does not apply.
It is plausible that their Exercise 
can be modified to hold for polytopes that are not $L$-dimensional, 
or polytopes that
are exactly of the type of $P_n$ above, to yield (2).

\subsection{Are the $b_i$'s Small?}

For the coin sets $\{1,5,10,25\}$, $\{2,3,4\}$, $\{3,5,6\}$, 
and $\{1,4,6,11\}$
the $b_i$'s are all in $[-0.4277,1.3636]$.
The smallest difference between the $b_i$'s is $1.0962$
and the largest difference is $1.4277$.

One conjecture is that there is some constant $B$ such that for {\it all} coin sets
the $b_i$'s are in $[-B,B]$. Another conjecture is that there is some slow growing function
$h(L,a_1,\ldots,a_L)$ such that for the coin set $a_1,\ldots,a_L$
all of the $b_i$'s are in $[-h(L,a_1,\ldots,a_L),h(L,a_1,\ldots,a_L)]$.
Similar conjectures can be made for the difference.

All of the coin sets above have no repeated coins. For the coin set $\{1,19,19,20\}$
the smallest $b_i$ is -6.3644 and the largest $b_i$ is 7.0953,
for a difference of $13.4597$.
It may be that such coin sets behave very differently.
Hence we only make the above conjectures for coin sets where all of the coins are distinct.

\section{Acknowledgment}

We would like to thank Daniel Smolyak,
Larry Washington, Sam Zbarsky for
proofreading and discussion.
We would like to thank Matthias Beck and 
Michele Vergne for pointing 
us to the rich literature
of the change problem and for many enlightening email exchanges.

We would also like to thank the referees. They made comments that
improved the paper considerably. In particular, the proofs in
Sections~\ref{se:lemmas} are much
improved, and the proof of Theorem~\ref{th:main} is somewhat
less cumbersome.

\newpage

\appendix
\section{$h_r$ Polynomials for $\{1,5,10,25\}$}\label{ap:h151025}

$M=LCM(1,5,10,25)=50$.

\[ 
 \begin{array}{rl|rl}
h_{0}(x) & = \frac{1}{7500}x^{3} + \frac{9}{1000}x^{2} + \frac{53}{300}x + 1& 
 h_{5}(x) & = \frac{1}{7500}x^{3} + \frac{9}{1000}x^{2} + \frac{53}{300}x + \frac{7}{8}\\
h_{1}(x) & = \frac{1}{7500}x^{3} + \frac{43}{5000}x^{2} + \frac{1193}{7500}x + \frac{4161}{5000}& 
 h_{6}(x) & = \frac{1}{7500}x^{3} + \frac{43}{5000}x^{2} + \frac{1193}{7500}x + \frac{442}{625}\\
h_{2}(x) & = \frac{1}{7500}x^{3} + \frac{41}{5000}x^{2} + \frac{1067}{7500}x + \frac{426}{625}& 
 h_{7}(x) & = \frac{1}{7500}x^{3} + \frac{41}{5000}x^{2} + \frac{1067}{7500}x + \frac{2783}{5000}\\
h_{3}(x) & = \frac{1}{7500}x^{3} + \frac{39}{5000}x^{2} + \frac{947}{7500}x + \frac{2737}{5000}& 
 h_{8}(x) & = \frac{1}{7500}x^{3} + \frac{39}{5000}x^{2} + \frac{947}{7500}x + \frac{264}{625}\\
h_{4}(x) & = \frac{1}{7500}x^{3} + \frac{37}{5000}x^{2} + \frac{833}{7500}x + \frac{268}{625}& 
 h_{9}(x) & = \frac{1}{7500}x^{3} + \frac{37}{5000}x^{2} + \frac{833}{7500}x + \frac{1519}{5000}\\
\end{array} 
\]
\[ 
 \begin{array}{rl|rl}
h_{10}(x) & = \frac{1}{7500}x^{3} + \frac{9}{1000}x^{2} + \frac{53}{300}x + \frac{6}{5}& 
 h_{15}(x) & = \frac{1}{7500}x^{3} + \frac{9}{1000}x^{2} + \frac{53}{300}x + \frac{7}{8}\\
h_{11}(x) & = \frac{1}{7500}x^{3} + \frac{43}{5000}x^{2} + \frac{1193}{7500}x + \frac{5161}{5000}& 
 h_{16}(x) & = \frac{1}{7500}x^{3} + \frac{43}{5000}x^{2} + \frac{1193}{7500}x + \frac{442}{625}\\
h_{12}(x) & = \frac{1}{7500}x^{3} + \frac{41}{5000}x^{2} + \frac{1067}{7500}x + \frac{551}{625}& 
 h_{17}(x) & = \frac{1}{7500}x^{3} + \frac{41}{5000}x^{2} + \frac{1067}{7500}x + \frac{2783}{5000}\\
h_{13}(x) & = \frac{1}{7500}x^{3} + \frac{39}{5000}x^{2} + \frac{947}{7500}x + \frac{3737}{5000}& 
 h_{18}(x) & = \frac{1}{7500}x^{3} + \frac{39}{5000}x^{2} + \frac{947}{7500}x + \frac{264}{625}\\
h_{14}(x) & = \frac{1}{7500}x^{3} + \frac{37}{5000}x^{2} + \frac{833}{7500}x + \frac{393}{625}& 
 h_{19}(x) & = \frac{1}{7500}x^{3} + \frac{37}{5000}x^{2} + \frac{833}{7500}x + \frac{1519}{5000}\\
\end{array} 
\]
\[ 
 \begin{array}{rl|rl}
h_{20}(x) & = \frac{1}{7500}x^{3} + \frac{9}{1000}x^{2} + \frac{53}{300}x + \frac{4}{5}& 
 h_{25}(x) & = \frac{1}{7500}x^{3} + \frac{9}{1000}x^{2} + \frac{53}{300}x + \frac{7}{8}\\
h_{21}(x) & = \frac{1}{7500}x^{3} + \frac{43}{5000}x^{2} + \frac{1193}{7500}x + \frac{3161}{5000}& 
 h_{26}(x) & = \frac{1}{7500}x^{3} + \frac{43}{5000}x^{2} + \frac{1193}{7500}x + \frac{442}{625}\\
h_{22}(x) & = \frac{1}{7500}x^{3} + \frac{41}{5000}x^{2} + \frac{1067}{7500}x + \frac{301}{625}& 
 h_{27}(x) & = \frac{1}{7500}x^{3} + \frac{41}{5000}x^{2} + \frac{1067}{7500}x + \frac{2783}{5000}\\
h_{23}(x) & = \frac{1}{7500}x^{3} + \frac{39}{5000}x^{2} + \frac{947}{7500}x + \frac{1737}{5000}& 
 h_{28}(x) & = \frac{1}{7500}x^{3} + \frac{39}{5000}x^{2} + \frac{947}{7500}x + \frac{264}{625}\\
h_{24}(x) & = \frac{1}{7500}x^{3} + \frac{37}{5000}x^{2} + \frac{833}{7500}x + \frac{143}{625}& 
 h_{29}(x) & = \frac{1}{7500}x^{3} + \frac{37}{5000}x^{2} + \frac{833}{7500}x + \frac{1519}{5000}\\
\end{array} 
\]
\[ 
 \begin{array}{rl|rl}
h_{30}(x) & = \frac{1}{7500}x^{3} + \frac{9}{1000}x^{2} + \frac{53}{300}x + 1& 
 h_{35}(x) & = \frac{1}{7500}x^{3} + \frac{9}{1000}x^{2} + \frac{53}{300}x + \frac{43}{40}\\
h_{31}(x) & = \frac{1}{7500}x^{3} + \frac{43}{5000}x^{2} + \frac{1193}{7500}x + \frac{4161}{5000}& 
 h_{36}(x) & = \frac{1}{7500}x^{3} + \frac{43}{5000}x^{2} + \frac{1193}{7500}x + \frac{567}{625}\\
h_{32}(x) & = \frac{1}{7500}x^{3} + \frac{41}{5000}x^{2} + \frac{1067}{7500}x + \frac{426}{625}& 
 h_{37}(x) & = \frac{1}{7500}x^{3} + \frac{41}{5000}x^{2} + \frac{1067}{7500}x + \frac{3783}{5000}\\
h_{33}(x) & = \frac{1}{7500}x^{3} + \frac{39}{5000}x^{2} + \frac{947}{7500}x + \frac{2737}{5000}& 
 h_{38}(x) & = \frac{1}{7500}x^{3} + \frac{39}{5000}x^{2} + \frac{947}{7500}x + \frac{389}{625}\\
h_{34}(x) & = \frac{1}{7500}x^{3} + \frac{37}{5000}x^{2} + \frac{833}{7500}x + \frac{268}{625}& 
 h_{39}(x) & = \frac{1}{7500}x^{3} + \frac{37}{5000}x^{2} + \frac{833}{7500}x + \frac{2519}{5000}\\
\end{array} 
\]
\[ 
 \begin{array}{rl|rl}
h_{40}(x) & = \frac{1}{7500}x^{3} + \frac{9}{1000}x^{2} + \frac{53}{300}x + 1& 
 h_{45}(x) & = \frac{1}{7500}x^{3} + \frac{9}{1000}x^{2} + \frac{53}{300}x + \frac{27}{40}\\
h_{41}(x) & = \frac{1}{7500}x^{3} + \frac{43}{5000}x^{2} + \frac{1193}{7500}x + \frac{4161}{5000}& 
 h_{46}(x) & = \frac{1}{7500}x^{3} + \frac{43}{5000}x^{2} + \frac{1193}{7500}x + \frac{317}{625}\\
h_{42}(x) & = \frac{1}{7500}x^{3} + \frac{41}{5000}x^{2} + \frac{1067}{7500}x + \frac{426}{625}& 
 h_{47}(x) & = \frac{1}{7500}x^{3} + \frac{41}{5000}x^{2} + \frac{1067}{7500}x + \frac{1783}{5000}\\
h_{43}(x) & = \frac{1}{7500}x^{3} + \frac{39}{5000}x^{2} + \frac{947}{7500}x + \frac{2737}{5000}& 
 h_{48}(x) & = \frac{1}{7500}x^{3} + \frac{39}{5000}x^{2} + \frac{947}{7500}x + \frac{139}{625}\\
h_{44}(x) & = \frac{1}{7500}x^{3} + \frac{37}{5000}x^{2} + \frac{833}{7500}x + \frac{268}{625}& 
 h_{49}(x) & = \frac{1}{7500}x^{3} + \frac{37}{5000}x^{2} + \frac{833}{7500}x + \frac{519}{5000}\\
\end{array} 
\]


\section{$h_r'$ Polynomials for $\{1,5,10,25\}$}\label{ap:hr151025}

$M'=LCM(GCD(5,10),GCD(5,25),GCD(10,25))=LCM(5,5,5)=5$.

\[ 
 \begin{array}{rl}
h_{0}'(x) & = \frac{1}{7500}x^{3} + \frac{9}{1000}x^{2} + \frac{53}{300}x \cr
h_{1}'(x) & = \frac{1}{7500}x^{3} + \frac{43}{5000}x^{2} + \frac{1193}{7500}x \cr
h_{2}'(x) & = \frac{1}{7500}x^{3} + \frac{41}{5000}x^{2} + \frac{1067}{7500}x \cr
h_{3}'(x) & = \frac{1}{7500}x^{3} + \frac{39}{5000}x^{2} + \frac{947}{7500}x \cr
h_{4}'(x) & = \frac{1}{7500}x^{3} + \frac{37}{5000}x^{2} + \frac{833}{7500}x \cr
\end{array}
\]

The smallest $b_i$ is $\frac{519}{5000}=0.1038$ and the largest $b_i$ is $\frac{6}{5}=1.2$. 
The difference between the largest and smallest is 1.0962.

\section{$h_r$ Polynomials for $\{2,3,4\}$}\label{ap:h234}

$M=LCM(2,3,4)=12$

\[ 
 \begin{array}{rl|rl} 
h_{0}(x) & = \frac{1}{48}x^{2} + \frac{1}{4}x + 1&
 h_{4}(x) & = \frac{1}{48}x^{2} + \frac{1}{8}x - \frac{7}{48} \\ 
h_{1}(x) & = \frac{1}{48}x^{2} + \frac{1}{8}x - \frac{7}{48}&
 h_{5}(x) & = \frac{1}{48}x^{2} + \frac{1}{4}x + \frac{3}{4} \\ 
h_{2}(x) & = \frac{1}{48}x^{2} + \frac{1}{4}x + \frac{5}{12}&
 h_{6}(x) & = \frac{1}{48}x^{2} + \frac{1}{8}x + \frac{5}{48} \\ 
h_{3}(x) & = \frac{1}{48}x^{2} + \frac{1}{8}x + \frac{7}{16}&
 h_{7}(x) & = \frac{1}{48}x^{2} + \frac{1}{4}x + \frac{2}{3} \\ 
\end{array} 
\] 
\[ 
 \begin{array}{rl|rl} 
h_{8}(x) & = \frac{1}{48}x^{2} + \frac{1}{4}x + \frac{2}{3}&
 h_{12}(x) & = \frac{1}{48}x^{2} + \frac{1}{8}x - \frac{7}{48} \\ 
h_{9}(x) & = \frac{1}{48}x^{2} + \frac{1}{8}x + \frac{3}{16}&
 h_{13}(x) & = \frac{1}{48}x^{2} + \frac{1}{4}x + \frac{5}{12} \\ 
h_{10}(x) & = \frac{1}{48}x^{2} + \frac{1}{4}x + \frac{5}{12}&
 h_{14}(x) & = \frac{1}{48}x^{2} + \frac{1}{8}x + \frac{7}{16} \\ 
h_{11}(x) & = \frac{1}{48}x^{2} + \frac{1}{8}x + \frac{5}{48}&
 h_{15}(x) & = \frac{1}{48}x^{2} + \frac{1}{4}x + \frac{2}{3} \\ 
\end{array} 
\] 

\section{$h_r'$ Polynomials for $\{2,3,4\}$}\label{se:hp234}

$M'=LCM(GCD(2,3),GCD(2,4),GCD(3,4))=LCM(1,2,1)=2$.

\[
\begin{array}{rl}
h_{0}'(x) & = \frac{1}{48}x^{2} + \frac{1}{4}x \cr
h_{1}'(x) & = \frac{1}{48}x^{2} + \frac{1}{8}x \cr
\end{array}
\]


The smallest $b_i$ is $-\frac{7}{48}=-0.1458$ and the largest $b_i$ is $1$.
The difference between the largest and smallest is $1.1458$.


\newpage

\section{$h_r$ Polynomials for $\{3,5,6\}$}\label{ap:h356}

$M=LCM(3,5,6)=30$.

\[ 
 \begin{array}{rl|rl} 
h_{0}(x) & = \frac{1}{180}x^{2} + \frac{2}{15}x + 1&
 h_{9}(x) & = \frac{1}{180}x^{2} + \frac{7}{90}x + \frac{17}{36} \\ 
h_{1}(x) & = \frac{1}{180}x^{2} + \frac{1}{45}x - \frac{1}{36}&
 h_{10}(x) & = \frac{1}{180}x^{2} + \frac{2}{15}x + 1\\
h_{2}(x) & = \frac{1}{180}x^{2} + \frac{7}{90}x - \frac{8}{45}&
 h_{11}(x) & = \frac{1}{180}x^{2} + \frac{1}{45}x - \frac{77}{180} \\ 
h_{3}(x) & = \frac{1}{180}x^{2} + \frac{2}{15}x + \frac{11}{20}&
 h_{12}(x) & = \frac{1}{180}x^{2} + \frac{7}{90}x + \frac{1}{45} \\ 
h_{4}(x) & = \frac{1}{180}x^{2} + \frac{1}{45}x - \frac{8}{45}&
 h_{13}(x) & = \frac{1}{180}x^{2} + \frac{2}{15}x + \frac{7}{20} \\ 
h_{5}(x) & = \frac{1}{180}x^{2} + \frac{7}{90}x + \frac{17}{36}&
 h_{14}(x) & = \frac{1}{180}x^{2} + \frac{1}{45}x + \frac{2}{9} \\ 
h_{6}(x) & = \frac{1}{180}x^{2} + \frac{2}{15}x + 1&
 h_{15}(x) & = \frac{1}{180}x^{2} + \frac{7}{90}x + \frac{17}{36} \\ 
h_{7}(x) & = \frac{1}{180}x^{2} + \frac{1}{45}x - \frac{77}{180}&
 h_{16}(x) & = \frac{1}{180}x^{2} + \frac{2}{15}x + \frac{3}{5} \\ 
h_{8}(x) & = \frac{1}{180}x^{2} + \frac{7}{90}x + \frac{1}{45}&
 h_{17}(x) & = \frac{1}{180}x^{2} + \frac{1}{45}x - \frac{41}{180} \\ 
\end{array} 
\] 

\[ 
 \begin{array}{rl|rl} 
h_{18}(x) & = \frac{1}{180}x^{2} + \frac{2}{15}x + \frac{4}{5}&
 h_{27}(x) & = \frac{1}{180}x^{2} + \frac{7}{90}x + \frac{49}{180} \\ 
h_{19}(x) & = \frac{1}{180}x^{2} + \frac{1}{45}x - \frac{77}{180}&
 h_{28}(x) & = \frac{1}{180}x^{2} + \frac{2}{15}x + \frac{3}{5} \\ 
h_{20}(x) & = \frac{1}{180}x^{2} + \frac{7}{90}x + \frac{2}{9}&
 h_{29}(x) & = \frac{1}{180}x^{2} + \frac{1}{45}x - \frac{1}{36} \\ 
h_{21}(x) & = \frac{1}{180}x^{2} + \frac{2}{15}x + \frac{3}{4}&
 h_{30}(x) & = \frac{1}{180}x^{2} + \frac{7}{90}x + \frac{2}{9} \\ 
h_{22}(x) & = \frac{1}{180}x^{2} + \frac{1}{45}x - \frac{8}{45}&
 h_{31}(x) & = \frac{1}{180}x^{2} + \frac{2}{15}x + \frac{7}{20} \\ 
h_{23}(x) & = \frac{1}{180}x^{2} + \frac{7}{90}x + \frac{49}{180}&
 h_{32}(x) & = \frac{1}{180}x^{2} + \frac{1}{45}x + \frac{1}{45} \\ 
h_{24}(x) & = \frac{1}{180}x^{2} + \frac{2}{15}x + \frac{3}{5}&
 h_{33}(x) & = \frac{1}{180}x^{2} + \frac{7}{90}x + \frac{13}{180} \\ 
h_{25}(x) & = \frac{1}{180}x^{2} + \frac{1}{45}x - \frac{1}{36}&
 h_{34}(x) & = \frac{1}{180}x^{2} + \frac{2}{15}x + 1\\
h_{26}(x) & = \frac{1}{180}x^{2} + \frac{7}{90}x + \frac{2}{9}&
 h_{35}(x) & = \frac{1}{180}x^{2} + \frac{1}{45}x - \frac{1}{36} \\ 
\end{array} 
\] 

\section{$h_r'$ Polynomials for $\{3,5,6\}$}\label{ap:hp356}

$M'=LCM(GCD(3,5),GCD(3,6),GCD(5,6))=LCM(1,3,1)=3$.

\[
\begin{array}{rl}
h_{0}(x) & = \frac{1}{180}x^{2} + \frac{2}{15}x\cr
h_{1}(x) & = \frac{1}{180}x^{2} + \frac{1}{45}x\cr
h_{2}(x) & = \frac{1}{180}x^{2} + \frac{7}{90}x\cr
h_{3}(x) & = \frac{1}{180}x^{2} + \frac{2}{15}x\cr
\end{array}
\]

The smallest $b_i$ is $-\frac{77}{180}=-0.4277$ and the largest $b_i$ is $1$.
The difference between the largest and smallest is $1.4277$.



\newpage

\section{Some of the $h_r$ Polynomials for $\{1,4,6,11\}$}\label{ap:h14611}

$M=LCM(1,4,6,11)=264$

\[ 
 \begin{array}{rl}
 h_{22}(x) & = \frac{1}{1584}x^{3} + \frac{1}{48}x^{2} + \frac{101}{528}x - \frac{9}{176}\cr
 h_{87}(x) & = \frac{1}{1584}x^{3} + \frac{1}{48}x^{2} + \frac{101}{528}x - \frac{9}{176}\cr
 h_{98}(x) & = \frac{1}{1584}x^{3} + \frac{1}{48}x^{2} + \frac{7}{33}x - \frac{23}{396}\cr
\end{array} 
\]

\section{$h_r'$ Polynomials for $\{1,4,6,11\}$}\label{ap:hr14611}

$M'=LCM(GCD(4,6),GCD(4,11),GCD(6,11))=LCM(2,1,1)=2$.

\[ 
 \begin{array}{rl}
h_{0}'(x) & = \frac{1}{1584}x^{3} + \frac{1}{48}x^{2} + \frac{101}{528}x\cr
h_{1}'(x) & = \frac{1}{1584}x^{3} + \frac{1}{48}x^{2} + \frac{101}{528}x\cr
\end{array} 
\] 

The smallest $b_i$ is $-\frac{23}{396}=-0.05808$ and the largest $b_i$ is $\frac{15}{11}=1.3636$.
The difference between the largest and smallest is $1.4168$.


\section{Some of the $h_r$ Polynomials for $\{1,19,19,20\}$}\label{ap:h1191920}

$M=LCM(1,19,19,20)=380$.

For all $0\le k\le 19$
\[
 \begin{array}{rl}
h_{19k}(x)    & = \frac{1}{43320}x^{3} + \frac{59}{28880}x^{2} - \frac{1}{228}x + 1\cr
h_{19k+17}(x) & = \frac{1}{43320}x^{3} + \frac{59}{28880}x^{2} - \frac{1}{228}x + \frac{35}{16} \cr
h_{19k+18}(x) & = \frac{1}{43320}x^{3} + \frac{59}{28880}x^{2} - \frac{127}{4332}x + \frac{5279}{7220}\cr
\end{array}
\]

\section{$h_r'$ Polynomials for $\{1,19,19,20\}$}\label{ap:hp1191920}

$M'=LCM(GCD(19,19),GCD(19,20))=LCM(19,1)=19$.

\[ 
 \begin{array}{rl|rl} 
h_{0}'(x) & = \frac{1}{43320}x^{3} + \frac{59}{28880}x^{2} - \frac{1}{228}x + 1&
h_{5}'(x) & = \frac{1}{43320}x^{3} + \frac{59}{28880}x^{2} + \frac{341}{4332}x + \frac{3191}{5776} \\ 
h_{1}'(x) & = \frac{1}{43320}x^{3} + \frac{59}{28880}x^{2} + \frac{77}{4332}x + \frac{28307}{28880}&
h_{6}'(x) & = \frac{1}{43320}x^{3} + \frac{59}{28880}x^{2} + \frac{377}{4332}x + \frac{2883}{7220} \\ 
h_{2}'(x) & = \frac{1}{43320}x^{3} + \frac{59}{28880}x^{2} + \frac{161}{4332}x + \frac{6623}{7220}&
h_{7}'(x) & = \frac{1}{43320}x^{3} + \frac{59}{28880}x^{2} + \frac{401}{4332}x + \frac{7047}{28880} \\ 
h_{3}'(x) & = \frac{1}{43320}x^{3} + \frac{59}{28880}x^{2} + \frac{233}{4332}x + \frac{23671}{28880}&
h_{8}'(x) & = \frac{1}{43320}x^{3} + \frac{59}{28880}x^{2} + \frac{413}{4332}x + \frac{9}{95} \\ 
h_{4}'(x) & = \frac{1}{43320}x^{3} + \frac{59}{28880}x^{2} + \frac{293}{4332}x + \frac{251}{361}&
h_{9}'(x) & = \frac{1}{43320}x^{3} + \frac{59}{28880}x^{2} + \frac{413}{4332}x - \frac{233}{5776} \\ 
\end{array} 
\] 
\[ 
 \begin{array}{rl|rl} 
h_{10}'(x) & = \frac{1}{43320}x^{3} + \frac{59}{28880}x^{2} + \frac{401}{4332}x - \frac{221}{1444}&
h_{15}'(x) & = \frac{1}{43320}x^{3} + \frac{59}{28880}x^{2} + \frac{161}{4332}x - \frac{549}{5776} \\ 
h_{11}'(x) & = \frac{1}{43320}x^{3} + \frac{59}{28880}x^{2} + \frac{377}{4332}x - \frac{6793}{28880}&
h_{16}'(x) & = \frac{1}{43320}x^{3} + \frac{59}{28880}x^{2} + \frac{77}{4332}x + \frac{177}{1805} \\ 
h_{12}'(x) & = \frac{1}{43320}x^{3} + \frac{59}{28880}x^{2} + \frac{341}{4332}x - \frac{503}{1805}&
h_{17}'(x) & = \frac{1}{43320}x^{3} + \frac{59}{28880}x^{2} - \frac{1}{228}x + \frac{10707}{28880} \\ 
h_{13}'(x) & = \frac{1}{43320}x^{3} + \frac{59}{28880}x^{2} + \frac{293}{4332}x - \frac{7949}{28880}&
h_{18}'(x) & = \frac{1}{43320}x^{3} + \frac{59}{28880}x^{2} - \frac{127}{4332}x + \frac{5279}{7220} \\ 
h_{14}'(x) & = \frac{1}{43320}x^{3} + \frac{59}{28880}x^{2} + \frac{233}{4332}x - \frac{313}{1444}&
h_{19}'(x) & = \frac{1}{43320}x^{3} + \frac{59}{28880}x^{2} - \frac{1}{228}x + \frac{35}{16} \\ 
\end{array} 
\] 

The smallest $b_i$ is $-\frac{36731}{5776}=-6.3644$ and the largest $b_i$ is $\frac{12807}{1805}=7.0953$.
The difference between the largest and smallest is $13.4597$.



\end{document}